\documentclass[11pt]{amsart} 
\usepackage{amsmath, amssymb, amsthm} 
\usepackage{mathtools}
\usepackage{dutchcal}

\newtheorem{theorem}{Theorem} 

\newtheorem{corollary}{Corollary} 
 
\theoremstyle{definition} 
\newtheorem{definition}{Definition} 
\theoremstyle{remark} 

\begin{document} 

\title{Tail Bounds via Southwest Boundary} 
\author{Stephen Jordan Harrison} 
\date{\today}

\maketitle

\begin{abstract}
We derive upper bounds for probabilities of the form \\ \(P(g(\mathbf{X})\ge t)\) using the southwest boundary (recently introduced in our previous work)
\(\partial_{\mathrm{SW}} Q(g^{-1}[t,\infty))\), where \(Q\) is a reflection to the first quadrant. Under natural continuity, symmetry, and monotonicity assumptions on 
\(g\), this yields explicit and computable bounds of the form \(P(g(\mathbf{X})\ge t)\le ns_t\), where \(s_t\) is the unique parameter at which the line 
\(L(s)=(f_1^{-1}(s),\dots,f_n^{-1}(s))\) intersects the southwest boundary. In particular, when \(g\) is a homogeneous polynomial of degree \(k\) (plus a constant $C$) and 
all tail bounds on the random variables are identical, the bound proves to the closed-form expression
\[
P(g(\mathbf{X})\ge t)\leq nf\bigg(\frac{(t-C)^{1/k}}{(\sum_i|a_i|)^{1/k}}\bigg)
\]
where $a_i$ are the coefficients of the monomials in $g$. We then obtain an explicit tail bound for the trace of a Schur multiplier acting on random matrices with 
identical tail bounds on the random variables. No assumptions are made about independence or dependence.
\end{abstract}

\section{Introduction}
If $X=(X_1,...,X_n)$ is random vector in $\mathbb{R}^n$ with the $X_i$ possibly dependent such that 
$$
P(|X_i|\geq t) \leq f_i(t), \ \forall t \geq 0
$$
for all $i$, write $X \in \mathcal{T}^d(f)$ where $f=(f_1,...,f_n).$\footnote{In \cite{harrison2026} $\mathcal{T}^d(f)$ was instead a set (class) of measures on 
$\mathbb{R}^n$ satisfying certain tail bound conditions; it is easy to convert from measures to r.v.'s and vice versa.} 

For $x=(x_1,...,x_n), y=(y_1,...,y_n) \in [0,\infty)^n$ define $x\leq' y \iff x_i \leq y_i$ for all $i$. If $\emptyset \neq V \subset [0,\infty)^n$ is closed define 
the \textit{southwest boundary} of $V$ as 
\begin{equation}\label{SWdefinition}
\partial_{SW}V = \{x \in V: \nexists x' \in V, x' <' x\}.
\end{equation}
That is, $\partial_{SW}V$ is the set of minimal elements of $V$. Let 
\begin{align*}
\mathcal{B} &= \big\{f:[0,\infty) \rightarrow [0,1]: f(0) = 1, f \downarrow 0,  f \text{ is left continuous} \big\} \\
\mathcal{B}_c &= \big\{f: [0,\infty) \rightarrow (0,\infty): f \text{ is continuous}, \text{ strictly decreasing}, f(0) \geq 1, \\ 
                                                                                & \ \ \ \ \ \ \ \ \ \ \ \ \ \ \ \ \ \ \ \ \ \ \ \ \ \ \ \ \ \ \ \ \ \ \ \ \ \ \ \ \ \ \ \ \ \ \ \ \ \ \ \ \ \ \ \ \ \ \ \ \ \ \ \ \ \ \ \
                                                                                 \text{ and} \lim_{t\rightarrow \infty} f(t) = 0\big\}.
\end{align*}
In \cite{harrison2026} we proved if $\emptyset \neq V \subset \mathbb{R}^n$ is closed and $f_i \in \mathcal{B}$ for all $i$ then 
\begin{equation}\label{mainSWresult}
\sup_{X \in \mathcal{T}^d(f)} P(X \in V) = \sup_{X \in \mathcal{T}^d(f)} P\big(X \in \partial_{SW}Q(V)\big)
\end{equation}
where $Q(x_1,...,x_n) = (|x_1|,...,|x_n|), \forall x_i \in \mathbb{R}$. This was done by reflecting $X$ ($X$'s pushforward measure on $\mathbb{R}^n$) into $[0,\infty)^n$ via $Q$ and then retracting $X's$ mass ($X$'s pushforward measure in $\mathbb{R}^n$) 
in $Q(V)$ to $\partial_{SW}Q(V)$ while respecting the tail bounds. That is, for $X \in \mathcal{T}^d(f)$ there exists $Y_X \in \mathcal{T}^d(f)$ such that
\begin{equation}\label{secondSWresult}
P(X \in V) = P\big(Y_X \in \partial_{SW}Q(V)\big).
\end{equation}
\eqref{mainSWresult} is relevant because if $g:\mathbb{R}^n \rightarrow \mathbb{R}$ is continuous then
$$
\sup_{X \in \mathcal{T}^d(f)} P(g(X) \geq t) = \sup_{X \in \mathcal{T}^d(f)} P\big(X \in g^{-1}[t,\infty)\big)
$$
and $g^{-1}[t,\infty)$ is closed. So a bound may potentially be placed on $P(g(X) \geq t)$ by analyzing $\partial_{SW} Q\big(g^{-1}[t,\infty)\big)$, and 
$\partial_{SW} Q\big(g^{-1}[t,\infty)\big)$ is detailed enough to place the best possible bound on $P(g(X) \geq t)$ if analyzed exactly. 
Analyzing $\partial_{SW} Q\big(g^{-1}[t,\infty)\big)$ (non-exactly) will be our goal in this article. 

For the following results in sections \ref{SectionGeneralBounds}-\ref{SectionSchurmultiplier}, no assumptions are made about the independence or dependence of the r.v.'s.
\section{General Bounds}\label{SectionGeneralBounds}
Note \eqref{mainSWresult} and \eqref{secondSWresult} also holds if $f=(f_1,...,f_n) \in \mathcal{B}_c^n$ since for all $i$, $f_i':= \min\{f_i,1\} \in \mathcal{B}$ and if 
$f' = (f_1',...,f_n')$ then $X \in \mathcal{T}^d(f) \iff X \in \mathcal{T}^d(f')$.  

For the remainder of sections \ref{SectionGeneralBounds} and \ref{SectionHomogeneousPolynomial} let $f = (f_1,...,f_n) \in \mathcal{B}_c^n$. Consider the line 
\begin{equation}\label{Leq}
L(s) = (f_1^{-1}(s),...,f_n^{-1}(s)) \subset [0,\infty)^n
\end{equation}
for $s \in (0,1]$ which is well defined since $f_i(0) \geq 1$ for all $i$, and $f_i$ are continuous and strictly decreasing to $0$.

\begin{theorem}\label{theorem1} 
Let $X \in \mathcal{T}^d(f)$ and $g:\mathbb{R}^n \rightarrow \mathbb{R}$ be continuous. Suppose there exists $s_t > 0$ such that
$$
L(s_t) \in \partial_{SW} Q\big(g^{-1}[t,\infty)\big).
$$
Then
$$
P(g(X) \ge t) \le n s_t.
$$
\end{theorem}
\begin{proof} Let $x^* = L(s_t)$.

For each $k = 1, \dots, n$, define the region
$$
R_k := \{ x \in [0,\infty)^n : \pi_k(x) \ge \pi_k(x^*) \}.
$$
We claim
\begin{equation}\label{V_tsubset}
\partial_{SW}V_t \subseteq \bigcup_{k=1}^n R_k.
\end{equation}
Indeed, $x^* \in \text{RHS}$. And if $x^*\neq x \in \partial_{SW}V_t$ then some component of $x$ must be greater than the corresponding component of $x^*$,
otherwise $x^*$ would not be minimal in $V$. In particular $x \in \text{RHS}$ of $\eqref{V_tsubset}$. 

Now, let $Y_X$ be given by \eqref{secondSWresult}. Then 
\begin{align*}
P\big(X \in g^{-1}[t,\infty)\big) &= P\big(Y_X \in \partial_{SW}Q(g^{-1}[t,\infty)\big) \\
                                  &\leq P\Big(Y_X \in \bigcup_{k=1}^n R_k\Big)\\
                                  &\leq\sum_{k=1}^n P(Y_X \in R_k) \\ 
                                  &=\sum_{k=1}^n P(\pi_k(Y_X) \geq \pi_k(x^*)) \\ 
                                  &\leq \sum_{k=1}^n f_k(\pi_k(x^*)) \ \ \ \ \ \ \ (\text{since } Y_X \in \mathcal{T}^d(f)) \\ 
                                  &=n s_t
\end{align*}
since $\pi_k(x^*) = f_k^{-1}(s_t)$ for all $k$.
\end{proof}

\begin{theorem}\label{theorem2}
Let $X \in \mathcal{T}^d(f)$ and $h(s) = g\big(f_1^{-1}(s), \dots, f_n^{-1}(s)\big), \ \forall s\in(0,1]$. Suppose $\lim_{s\downarrow 0} h(s) = \infty$, $g:\mathbb{R}^n \rightarrow \mathbb{R}$ is continuous, even in 
each coordinate, and if any $n-1$ coordinates are held fixed, $g$ is strictly increasing in the remaining coordinate on $[0,\infty)$. Then 
$\forall t \geq h(1)$,
\begin{equation}\label{theorem2eq1}
P(g(X) \ge t) \leq n h^{-1}(t).
\end{equation}
Furthermore, if $f_1(0)=\cdots =f_n(0)$ then \eqref{theorem2eq1} holds $\forall t \geq g(\mathbf{0})$.
\end{theorem}
\begin{proof} 
Since $h$ is continuous, $h(1) \leq t$, and $\lim_{s\downarrow 0} h(s) = \infty$, by the intermediate value theorem, $h(s)=t$ has a solution. 
Clearly $h$ is strictly decreasing (since $f_i^{-1}$ are strictly decreasing), so this solution is unique. Now let $V_t = Q\big(g^{-1}[t,\infty)\big)$. Note by $g$'s evenness property that 
$$
V_t = \{x \in [0,\infty)^n: g(x) \geq t\}.
$$
We claim 
\begin{equation}\label{theorem2eq2}
\partial_{SW}V_t = g^{-1}(t)\cap [0,\infty)^n
\end{equation}
from which \eqref{theorem2eq1} immediately follows from theorem~\ref{theorem1}, since $h(s)=t$ implies $L(s) \in g^{-1}(t)$ \eqref{Leq} which implies 
$L(s) \in g^{-1}(t)\cap [0,\infty)^n$ since $L(s):(0,1] \rightarrow [0,\infty)^n$. Firstly by $g$'s properties $g(\mathbf{0}) \leq g(L(s)) = t$, so assuming  
$(x_1,...,x_n) \in V_t\backslash g^{-1}(t)$ it follows $(x_1,...,x_n) \subset [0,\infty)^n \backslash \{\mathbf{0}\}$. Then some $x_i$ is greater than $0$, say $x_1$, 
and furthermore by $g$'s properties $\exists \epsilon>0$ with $t<g(x_1-\epsilon,x_2,...,x_n)<g(x_1,...,x_n)$. In particular 
$V_t \ni (x_1-\epsilon,x_2,...,x_n) \leq' (x_1,...,x_n)$ which implies $(x_1,...,x_n) \notin \partial_{SW}V_t$ by \eqref{SWdefinition}. So 
$$
\partial_{SW}V_t \subset g^{-1}(t).
$$
Since $\partial_{SW}V_t \subset [0,\infty)^n$ this shows $\partial_{SW}V_t \subset g^{-1}(t)\cap [0,\infty)^n$. 
Conversely if $x \in g^{-1}(t)\cap [0,\infty)^n$ and $\mathbf{0} \leq' x' <' x$ then by $g$'s properties $g(x')<g(x)=t$ so $x' \notin V_t$. 
So $g^{-1}(t)\cap [0,\infty)^n \subset \partial_{SW}V_t$. 
This shows \eqref{theorem2eq2}. Hence \eqref{theorem2eq1} holds. 

Now, if $f_1(0)=\cdots = f_n(0) = c$, note $h(s)$ is definable on the larger domain $[c,0)$ and $h$ will still be continuous and strictly decreasing. 
Then $\forall t \geq h(c) = g(\mathbf{0})$, since $\lim_{s\downarrow 0} h(s) = \infty$, by the intermediate value theorem $h(s) = t$ has a unique solution. 
On the other hand, since $h\downarrow$, $h^{-1} \downarrow$ too. Then if $g(\mathbf{0}) \leq t < h(1)$,
$$
h^{-1}(t) \geq h^{-1}(h(1)) = 1
$$
so \eqref{theorem2eq1} also holds for $t \in [g(0),h(1))$. 
\end{proof}
\section{Homogeneous Polynomial Case}\label{SectionHomogeneousPolynomial}
One case is of particular interest. 
\begin{corollary}\label{homogeneouscorollary} Suppose $X \in \mathcal{T}^d(f)$ and $g$ is a homogeneous polynomial of degree $k$ plus a constant $C$. 
Then $\forall t \geq C$
\begin{equation}\label{homogeneouseq}
P(g(X) \geq t) \leq n f\left(\frac{(t-C)^{1/k}}{(\sum_i |a_i|)^{1/k}} \right)
\end{equation}
where $a_i$ are the coefficients of $g$. 
\end{corollary}

\begin{proof} 
Let $g(x) = g_1(x) + \cdots + g_m(x)+C$ where $g_i$ are monomials and $g'(x) = |g_1(x)|+\cdots+|g_m(x)|+C$. Then 
$$
P\big(g(X) \geq t\big) \leq P\big(|g(X)| \geq t\big) \leq P\big(g'(X) \geq t\big).
$$
$g'$ satisfies the properties of theorem~\ref{theorem2}. Applying it
\begin{equation}\label{homogeneouseq2}
P(g'(X)\geq t) \leq ns(t), \ \forall t \geq 0.
\end{equation}
where $s(t)$ is the unique solution to $h(s) := g'(f^{-1}(s),...,f^{-1}(s)) = t$. Solving for $s$ yields $h^{-1}(t) = s(t) =  f\bigg(\frac{(t-C)^{1/k}}{\left(\sum_i |a_i|\right)^{1/k}} \bigg)$.
So \eqref{homogeneouseq} holds. 
\end{proof}

\section{Application to Schur multipliers}\label{SectionSchurmultiplier}
\cite{Skripka2024} includes a definition of Schur multipliers, which we give here, and have already been studied at least once, in \cite{Skripka2024}.
\begin{definition}[Schur multiplier]
Let \(\mathcal{M}_n\) denote the space of \(n \times n\) real matrices and \(E_{i,j}\) the elementary matrix whose only non-zero entry is in position \((i,j)\). 
Let \(d,n \in \mathbb{N}\) and let
\[
\mathfrak{m}(d) = \{m_{j_1,\dots,j_{d+1}}\}_{j_1,\dots,j_{d+1}=1}^n
\]
be a multi-dimensional matrix with real entries. The \(d\)-linear Schur multiplier \(\mathfrak{M}_{\mathfrak{m}(d)} : \mathcal{M}_n \times \cdots \times \mathcal{M}_n 
\to \mathcal{M}_n\) is
\begin{equation}\label{Schurmultipliereq0}
\mathfrak{M}_{\mathfrak{m}(d)}(X_1,\dots,X_d) = \sum_{j_1,\dots,j_{d+1}=1}^n m_{j_1,\dots,j_{d+1}} x^{(1)}_{j_1,j_2} \cdots x^{(d)}_{j_d,j_{d+1}} E_{j_1,j_{d+1}}.
\end{equation}
where $x_{ij}^{(k)}$ is the $(i,j)$ entry of $X_k$.
\end{definition}
\begin{theorem}
Let $\mathfrak{M}_{m(d)}$ be a Schur multiplier and $X_1,...,X_d$ be random matrices. Write $X_{ij}^{(k)}$ for the $(i,j)$ entry of $X_k$. 
Let $f\in \mathcal{B}_c$. If 
$$
P\big(|X_{ij}^{(d)}| \geq t \big) \leq f(t), \ \forall t \geq 0, \ \forall i,j,k
$$
then $\forall t \geq 0$,
$$
P\bigg(\Big|\operatorname{Tr}(\mathfrak{M}_{\mathfrak{m}(d)}\big(X_1,\dots,X_d)\big)\Big| \ge t\bigg)\leq dn^2 f\Bigg(\frac{t^{1/d}}{||\mathfrak{m}(d)||_1^{1/d}}\Bigg).
$$
\end{theorem}
\begin{proof}
Apply corollary~\ref{homogeneouscorollary}.
\end{proof}
Again, this holds whether the r.v.'s are independent or dependent.\footnote{I originally developed theory for tail bounds in \cite{harrison2026} in the hopes of being able to apply it to Schur multipliers, so it is gratifying to finally have some 
confirming formula for Schur multipliers.}

\section{Comparison to other concentration inequalities}
The following concentration inequalities (tail bounds) are from \cite{Vershynin2023}, excluding possibly the union bound 
which I'm not aware is in \cite{Vershynin2023} or not.
\begin{theorem}[Hoeffding's inequality, two-sided]\label{Hoeffding1}
Let \(X_1, \dots, X_N\) be independent symmetric Bernoulli r.v.'s, and let \(a = (a_1, \dots, a_n) \in \mathbb{R}^N\). Then, for any \(t \ge 0\), we have
\[
P\left( \left| \sum_{i=1}^N a_i X_i \right| \ge t \right) \le 2 \exp\left( -\frac{t^2}{2 \|a\|_2^2} \right).
\]
\end{theorem}
In order to apply corollary~\ref{homogeneouscorollary} to Bernoulli r.v.'s $(X_i)_i$ one must prescribe a tail bound $f$ with $f(1)\geq P(|X_i|\geq 1) = 1$. Then corollary~\ref{homogeneouscorollary} gives
\begin{equation}\label{Hoeffding1eq}
P\left( \left| \sum_{i=1}^N a_i X_i \right| \ge t \right) \leq nf\big(t/||a||_1\big)
\end{equation}
which equals $n$ for $t=||a||_1$, so corollary~\ref{homogeneouscorollary} only produces a trivial bound here (since the LHS of \eqref{Hoeffding1eq} equals $0$ for $t>||a||_1$). 

Now, if \(X_i\) is a real r.v. such that \(P(|X_i| > t) \le 2e^{-c t^2}\), \(\forall t \ge 0\) one says \(X_i\) is \emph{sub-gaussian}. In this case, the sub-gaussian 
norm of \(X_i\) is defined as \(\|X_i\|_{\psi_2} = \inf\{ t > 0 : \mathbb{E} \exp(X_i^2/t^2) \le 2 \}\) which can be shown to be a norm. Then 

\begin{theorem}[General Hoeffding inequality]
Let \(X_1, \dots, X_N\) be independent mean-zero sub-gaussian r.v.'s (with the same $c$), and let \(a = (a_1, \dots, a_n) \in \mathbb{R}^N\). Then, for any 
\(t \ge 0\), we have
\begin{equation}\label{Hoeffding2eq}
P\left( \left| \sum_{i=1}^N a_i X_i \right| \ge t \right) \le 2 \exp\left( -\frac{c t^2}{K^2 \|a\|_2^2} \right)
\end{equation}
where \(K = \max_i \|X_i\|_{\psi_2}\).
\end{theorem}
In order to apply corollary~\ref{homogeneouscorollary} to the sub-gaussian case, it is necessary to instead consider $X_i$ such that \(P(|X_i| \geq t) \le 2e^{-c t^2}\) (note the change of $>$ to $\geq$).
This is because I only developed the southwest boundary for ``$\geq$" tail bounds. $>$ tail bounds were qualitatively different. In the case $P(|X_i|\geq t) \leq 2e^{-c t^2}$ 
for all $t\geq 0$, corollary~\ref{homogeneouscorollary} gives 
\begin{equation}\label{HomogeneousHoeffding}
P\left( \left| \sum_{i=1}^N a_i X_i \right| \ge t \right) \le 2n \exp\left( -\frac{c t^2}{\|a\|_1^2} \right).
\end{equation}
The $2n$ factor is worse. If $K \leq ||a||_1/||a||_2$ then \eqref{Hoeffding2eq} is better. If $K > ||a||_1/||a||_2$ then \eqref{HomogeneousHoeffding} is sharper for large $t$.

The union bound is the following observation: if $t_1\cdots t_n = t, \ t_i>0$ for all $i$, then if $|X_1\cdots X_n| \geq t$ necessarily $|X_i| \geq t_i$ for some $i$. Thus, 
$$
P\big(|X_1\cdots X_n| \geq t \big) \leq \sum_{k=1}^n P(|X_k|\geq t_k) \leq \sum_{k=1}^n f_k(t_k) 
$$
assuming $P(|X_k|\geq t) \leq f_k(t)$ for all $k$. If one assumes $f_1 = \cdots = f_n = f$ and guesses $t_1=\cdots = t_n = t^{1/n}$ might be optimal, 
one gets 
$$
P(|X_1\cdots X_n| \geq t) \leq nf(t^{1/n})
$$
which is the same bound given by applying corollary~\ref{homogeneouscorollary} to $g(x_1,...,x_n) = x_1\cdots x_n$. Thus theorem~\ref{theorem1} can be thought of as 
somewhat of a geometric generalization of the union bound. Also, the union bound idea gives
$$
P\big(\big|a_1X_1 + \cdots+ a_nX_n \big| \geq t\big) \leq nf\Big(\frac{t}{\sum_{i} |a_i|}\Big)
$$
(which is the same bound given by applying corollary~\ref{homogeneouscorollary} to $g(x_1,...,x_n) = a_1X_1+\cdots+a_nX_n$) and 
\begin{equation}\label{unionboundhomogeneous}
P(g(X) \geq t) \leq \bigg(\sum_{i=1}^m k_i \bigg) f\left(\frac{(t-C)^{1/k}}{(\sum_i |a_i|)^{1/k}} \right)
\end{equation}
where $g$ is from corollary~\ref{homogeneouscorollary}, $m$ is the number of monomials in $g$, and $k_i$ is the number of variables in the $i$th monomial not including powers. 
So this is where corollary~\ref{homogeneouscorollary} differs from the union bound, since $n \leq \sum_{i=1}^m k_i$ with inequality being strict 
if some r.v.\ occurs in more than one monomial.  

It is possible to use theorem~\ref{theorem2} to derive tail bounds on $||X||_p$ for $p \in (0,1)$, the Frobenius norm of a random matrix, and 
operator norm of a random matrix (since the operator norm is less than the Frobenius norm), however those inequalities can also be derived 
by a simple union bound. However if $A$ is a square random matrix then $g(A) = \det A$ is a homogeneous polynomial, and applying corollary~\ref{homogeneouscorollary} yields 
$$
P\big(|\det A| \geq t\big) \leq n^2 f\bigg(\frac{t^{1/n}}{(n!)^{1/n}}\bigg), \ \forall t \geq 0
$$
which is a significant improvement over the union bound, being smaller by a factor of $(n-1)!$. At this point its become somewhat clearer what we've inadvertently 
done is generalize the union bound method, in a way which avoids double counting separate appearances of the same r.v..

\end{document}